\documentclass{amsart}

 \usepackage{hyperref}
 \usepackage{amsmath}
 \usepackage{amsthm}
 \usepackage{amsfonts}
 \usepackage{amssymb}

\def\be{\begin{equation}}
\def\eea{\end{eqnarray}}
\def\ee{\end{equation}}
\def\bea{\begin{eqnarray}}
\def\ea{\end{array}}
\def\ba{\begin{array}}

\newcommand{\exval}[1]{\mbox{$\langle \, {#1}\, \rangle$}}
\newcommand{\bel}[1]{\begin{equation}\label{#1}}

\newcommand{\deriv}[1]{\mbox{$\displaystyle\frac{d}{d{#1}}$}}
\newcommand{\pderiv}[1]{\mbox{$\displaystyle\frac{\partial}{\partial{#1}}$}}

\theoremstyle{plain}
  \newtheorem{theorem}[subsection]{Theorem}
  \newtheorem{conjecture}[subsection]{Conjecture}

  \newtheorem{proposition}[subsection]{Proposition}
  
  \newtheorem{lemma}[subsection]{Lemma}
  
  \newtheorem{Distance conjecture}[subsection]{Distance Conjecture}
  \newtheorem{Bilinear distance conjecture}[subsection]{Bilinear Distance Conjecture}
  \newtheorem{Furstenburg problem}[subsection]{Furstenburg problem}
  \newtheorem{Discretized Furstenburg conjecture}[subsection]{Discretized Furstenburg 
  Conjecture}
  \newtheorem{Ring problem}[subsection]{Ring Problem}
  \newtheorem{Ring conjecture}[subsection]{Ring Conjecture}
  \newtheorem{Main theorem}[subsection]{Main Theorem}
  \newtheorem{Cauchy-Schwarz}[subsection]{Cauchy-Schwarz}

\theoremstyle{remark}
  \newtheorem{remark}[subsection]{Remark}

\theoremstyle{definition}

\include{psfig}

\begin{document}

\title[A solvability criterion for 
Navier-Stokes equations]
{A Solvability criterion  for  \\
Navier-Stokes equations
in  high dimensions}

\author{T. M. Viswanathan}
\address{  {Universidade Federal de Alagoas,
       Macei\'{o}--AL, CEP 57072-970, Brazil} }

\address{
{Consortium of the Americas for Interdisciplinary Science, 
University of New Mexico, 
800 Yale Blvd. NE, \\
Albuquerque, NM 87131, USA}}

\email{viswanathan.tenkasi@gmail.com}

\author{G. M. Viswanathan}
\address{
{Consortium of the Americas for Interdisciplinary Science, 
University of New Mexico, 
800 Yale Blvd. NE, \\
Albuquerque, NM 87131, USA}}

\address{  {Instituto de F\'{\i}sica, Universidade Federal de Alagoas,
       Macei\'{o}--AL, CEP 57072-970, Brazil} }

\email{Gandhi.Viswanathan@pq.cnpq.br}

\begin{abstract} 
We define the Ladyzhenskaya-Lions exponent $\alpha_{\rm \mbox{\tiny
    \sc l}} (n)=({2+n})/4$ for Navier-Stokes equations with
dissipation $-(-\Delta)^{\alpha}$ in ${\Bbb R}^n$, for all $n\geq 2$.
We review the proof of 
strong global solvability when $\alpha\geq \alpha_{\rm
  \mbox{\tiny \sc l}} (n)$, given smooth initial data.  If the
corresponding Euler equations for $n>2$ were to allow uncontrolled
growth of the enstrophy ${1\over 2} \|\nabla u \|^2_{L^2}$, then no
globally controlled coercive quantity is currently known to exist that
can regularize solutions of the Navier-Stokes equations for
$\alpha<\alpha_{\rm \mbox{\tiny \sc l}} (n)$. The energy is critical
under scale transformations only for $\alpha=\alpha_{\rm \mbox{\tiny
    \sc l}} (n)$.
\end{abstract}

\maketitle

\section{Introduction}

The Navier-Stokes equation for the velocity $u$ of an incompressible
fluid with viscous dissipation $-(-\Delta)^{\alpha}$ is given by
\bel{eq-ns}{\partial u \over \partial t} + u \cdot \nabla u +
\nabla p = -(-\Delta)^{\alpha} u, 
\ee 
where $u$ is a time-dependent divergence-free vector field in ${\Bbb
  R}^n$ with zero mean.  The pressure $p$ can be eliminated using
Leray projections, since it merely serves to ensure the divergence
free condition (see Equation (\ref{eq-p}) below).  Given the initial
condition
\bel{eq-init} u(x,0)=u_0(x) 
\ee 
where $u_0(x) \in C^{\infty}_c({\Bbb R}^n)$, the question is whether
or not solutions remain smooth.  

Olga Ladyzhenskaya~\cite{sk1,sk2}, in the 1960s, proved global
regularity for the two dimensional (\mbox{$n=2$}) parabolic
(\mbox{$\alpha=1$}) case.  For $n=3$ the challenge is to improve on
the well known exponent $\alpha\geq 5/4$ obtained long ago by
Jacques-Louis~Lions~\cite{lions1,lions2}, still the best known result
to date~\cite{new,katz,ms,epl}.  One decade ago, Mattingly and
Sinai~\cite{ms} gave an elementary proof of this result and showed
that there are no blow-ups and solutions remain smooth if $\alpha\geq
5/4$.  Later Katz and Pavlovi\'c~\cite{katz} extended the result by
proving additional statements, concerning the Hausdorff dimension of
the singular set at the time of first blow-up.  Recently, in the
context of theoretical physics, we have studied~\cite{epl} the problem
in higher dimensions, in order to gain a better understanding of
non-Newtonian turbulence and non-local anomalous diffusion of energy
density in the Fourier domain in hyper-dissipative incompressible
flows.

We draw the reader's attention to how the lowest value of the exponent
$\alpha$ known to ensure global regularity increases by $1/4=5/4-1$,
as the number of dimensions increases from $n=2$ to $n=3$. The
question that we address is whether or not this 1/4 increase per
dimension is a more general principle with deeper significance. Thus
motivated, we define the Ladyzhenskaya-Lions exponent
\bel{eq-lle}
\alpha_{\rm \mbox{\tiny \sc l}} (n)={2+n\over 4},~\forall~ n\geq 2 \;\;,
\ee
by linear extrapolation from $n=2$ and $n=3$.  We then show how to
generalize to all $n$ the seminal regularity results of Ladyzhenskaya
and Lions, starting with the following theorem:

\begin{theorem}\label{tm-1}
If $\alpha > \alpha_{\rm \mbox{\tiny \sc l}} (n)$ then one has strong global solvability for
the system (\ref{eq-ns}), (\ref{eq-init}) in ${\Bbb R}^n$.
\end{theorem}

In fact, we give two different proofs of this theorem.  Although it
may have gone somewhat unnoticed, Lions himself arrived at the result
40 years ago \cite{lions2} (see also the paper by Guermond and
Prudhomme \cite{new}).
Our own physical argument~\cite{epl} used
$L^1$ estimates in the Fourier domain to bound the rate of growth of
maximal norms of the velocity $u$ and its spatial derivatives. Since
then, we have been able to generalize to arbitrary $n$ an elegant
proof of Katz and Pavlovi\'c~\cite{katz} for the case $n=3$, based on
$L^2$ estimates.  We present first the latter proof, because it uses
standard methods of functional analysis.  We then present the second
proof, by reformulating in mathematical terms the physical arguments
given in ref.~\cite{epl}.  This second proof is longer, but which
nevertheless sheds further insight on how and why the cascades of
energy (and vorticity) continue to bedevil attempts to prove or
disprove global regularity for the important and well known special
case $\alpha=1$ and $n=3$.  For higher dimensions $n>3$ the
nonlinearity can lead to even more violent cascade of energy, such
that it becomes progressively harder to bound the uncontrolled growth
of the enstrophy, defined as $ {1\over 2} \| \nabla
u\|^2_{L^2}=-{1\over 2} \exval{u,\Delta u} $.

The two methods of proof, though equivalent, shed light on different
aspects of the problem and complement one another.  Taken together,
they seem to suggest that in higher dimensions $n\geq 3$ the
Ladyzhenskaya-Lions exponent represents a genuine critical
point~\cite{epl}.  Both methods rely ultimately (explicitly or
implicitly) on conservation of energy to regularize the solutions.
Solutions of the Navier-Stokes equations on a time interval $[0,T]$
satisfy
$$\|u(.,T)\|_{L^2}^2 = \|u_0\|_{L^2}^2 - \int_{0}^T \langle
(-\Delta)^{\alpha} u,u \rangle.$$ The second term on the right is
called the dissipation term.  The closely related Euler equations for
inviscid flow do not have the dissipation term.

We first remark on notation.  We will use $\langle,\rangle$ always to
denote an $L^2$ inner product in space. We follow the notation 
used in ref.~\cite{katz} throughout this paper for the expression $A
\lesssim B$, to mean $A \leq CB$ where $C$ is some constant.  This
constant may depend on $T$, the norms of the initial values, and on an
$\epsilon$ which we keep fixed throughout this paper. 

\subsection*{Acknowledgements}
We thank CNPq (Process 201809/2007-9) for funding. We thank Jean-Luc
Guermond, H. D. Jennings and V. Milman for helpful information.  We
also thank V.~M.~Kenkre and the Consortium of the Americas for
Interdisciplinary Science for their hospitality.

\section{Proof based on $L^2$ estimates}

The following is a generalization to ${\Bbb R}^n$ of the known proof
for $n=3$.  We briefly sketch the outline of the argument.  By taking
suitable inner products, one can bound the rate of growth of $L^p$
Sobolev norms of $u$. Then, by invoking a carefully chosen Sobolev
inequality, one can re-express these bounds in terms of $L^2$ norms,
at which point the rest of the proof proceeds conventionally and one
recovers global solvability.

The dependence on $n$ enters the final result solely through the
Sobolev inequality.  Most importantly, the dependence on $n$ of
$\alpha_{\rm \mbox{\tiny \sc l}} (n)$ offers a new and different
perspective on the well known explanation of why the parabolic
Navier-Stokes equations have global regularity in two but not in three
space dimensions.  Due to the close relationship between the Euler and
Navier-Stokes equations, the literature does not always make clear the
significant differences between them in two dimensions.  One
explanation argues that the for $n=2$ vorticity is conserved, whereas
for $n=3$ vorticity and hence enstrophy can grow.  This argument
actually misses the point for the Navier-Stokes case, since in two
dimensions the Euler equation itself has smooth solutions, rendering
dissipation irrelevant.  Plausible explanations have their basis on
energy considerations. For $n=2$ the direct energy cascade --- if it
counterfactually (i.e., hypothetically) existed in the Euler case ---
would remain weak enough that dissipation with $\alpha=1$ can control
it. But for $n>2$ the energy can cascade more violently, such that
dissipation with $\alpha=1$ cannot necessarily control it.

In this context, the following proof provides a clear interpretation,
viz. that Sobolev embedding becomes more ``costly'' for larger $n$.
By costly, we mean that the number of derivatives necessary for the
embedding increases with $n$.  Here, energy enters the picture
(via the norm $\| u \|^2_{L ^2}$), and the role of $n$ becomes clear.

\begin{proof}[Proof of Theorem \ref{tm-1}]

Let $H^{\beta}$ and  $W^{\beta,p}$ denote the $L^2$
and $L^p$ Sobolev spaces, respectively,  over ${\Bbb R}^n$ with
$\beta$ derivatives.
Taking the inner product by pairing 
Equation \eqref{eq-ns} with $u$, we get 
\bel{eq-energy} 
{1 \over 2} {\deriv{t}{ \|u\|_{L^2}^2}}
= -\| (-\Delta)^{{\alpha \over 2}} u\|_{L^2}^2 
\leq -\|u\|_{H^{\alpha}}^2 +\|u\|_{L^2}^2
\end{equation}
Hence, if the
solution $u$ remains smooth up to time $T$, we have the estimate
\bel{eq-easyestimate}
\int_0^T \|u\|_{H^{\alpha}}^2 dt \lesssim (1+T).
\end{equation}
We next pair Equation \eqref{eq-ns} with $(-\Delta)^{\beta}u$ in order to
estimate ${\partial(\|u\|_{H^{\beta}}^2) \over \partial t}$.
We obtain
\bel{eq-cleanpairing}{1 \over 2} {\deriv{t} \|(-\Delta)^{{\beta \over 2}}
u\|_{L^2}^2} + \langle u \cdot \nabla u,
(-\Delta)^{\beta} u \rangle = 
-\|(-\Delta)^{{\alpha + \beta \over 2}}u\|_{L^2}^2.
\end{equation}
We next estimate the nonlinear term \bel{eq-bal} \langle u \cdot
\nabla u,(-\Delta)^{\beta} u \rangle =\langle (-\Delta)^{{\beta \over
    2}}( u \cdot \nabla u), (-\Delta)^{{\beta \over 2}} u \rangle.
\ee Recall that $u$ is divergence free.  We can bound the absolute
value of the left hand side of Equation (\ref{eq-bal}) by
\bel{eq-bnd1} \| u \cdot \nabla u \|_{H^{\beta}} \|u\|_{H^{\beta}} \;
, \ee using the Cauchy-Schwarz inequality.  Next, we can apply the
following well known generalization~\cite{book-funcanal} of H\"older's
inequality:
$$
\|f_1 f_2 \|_{L^{p}} \leq \|f_1 \|_{L^{p_1}}  \|f_2 \|_{L^{p_2}} ,
\quad {1 \over p }= {1\over p_1} + {1\over p_2} \;\; .   
$$
So we can now bound (\ref{eq-bnd1}) by
\bel{eq-alt1}
\|u\|_{W^{0,p}} \|u\|_{W^{\beta+1,q}} \|u\|_{H^{\beta}} \;\; ,
\ee
or
equivalently, following ref~\cite{katz}, by
\bel{eq-alt2}
\|u\|_{W^{1,p}} \|u\|_{W^{\beta,q}} \|u\|_{H^{\beta}}  \;\; ,
\ee
with 
\bel{eq-sobolev-cond}
{1 \over 2} ={1 \over p} + {1 \over q}  \;\; . 
\ee

Now comes the crucial step, where $n$ enters
the picture. We twice invoke the
Sobolev embedding theorem to obtain, 
\bel{eq-sobolev}
\|u\|_{W^{0,p}} \|u\|_{W^{\beta+1,q}} \|u\|_{H^{\beta}} 
\lesssim \|u\|_{H^{\alpha}} \|u\|_{H^{\alpha+\beta}}
\|u\|_{H^{\beta}}\;\;,
\ee
where, using (\ref{eq-sobolev-cond}), the embedding is 
seen to be 
completely
continuous provided that
\bel{eq-condition}
2 \alpha-1>n/2 \;\; .
\ee
The embedding ``costs'' $n/2$ derivatives whereas  we 
can 
``spend''
 $2\alpha-1$ derivatives.  Remarkably, this is the only step of the
 argument involving $n$.  We note that the same bound  is obtained by
 using (\ref{eq-alt2}) instead of (\ref{eq-alt1}) in
 (\ref{eq-sobolev}).

 The rest of the argument proceeds straightforwardly.
From  Cauchy-Schwarz, we get,  
$$\|u\|_{H^{\alpha}} \|u\|_{H^{\alpha+\beta}}
\|u\|_{H^{\beta}} \leq \delta \|u\|_{H^{\alpha+\beta}}^2
+ {1 \over \delta}  \|u\|_{H^{\alpha}}^2 \|u\|_{H^{\beta}}^2.$$
Combining this with Equation  \eqref{eq-cleanpairing}, we get
$$ {\deriv{t} \|u\|_{H_{\beta}}^2}
\lesssim \|u\|_{H^{\alpha}}^2 \|u\|_{H^{\beta}}^2 + \|u\|_{L^2}^2.$$
Together  with \eqref{eq-easyestimate} and invoking 
Gronwall's inequality gives global solvability. 
The sole sufficient  condition (\ref{eq-condition}) can be rewritten as, 
\bel{eq-alpha-condition} 
\alpha> {2+n\over 4} \;\; .  
\ee
\end{proof}

\section{Proof based on $L^1$ (or $L^\infty$) estimates}

We next give a second proof corresponding to the original argument
presented in our earlier study, ref.~\cite{epl}.  The latter argument
has its basis on the analogy or connection~\cite{epl} between the
Navier-Stokes equation and nonlocal Fokker-Planck equations for
describing anomalous diffusion of ensembles of random walkers, such as
found in L\'evy flights~\cite{sokolov}.  Fourier transforming the
Euler equation converts the nonlinear term into a non-local one.
Taking an inner product of the Fourier transformed equation with the
Fourier transform of $u$, one obtains an equation which conserves
energy, in a manner analogous to how Fokker-Planck equations conserve
probability.  A well known fact about fractional or nonlocal
Fokker-Planck equations from the study of L\'evy flights concerns the
relationship between fat tails in the probability density function and
smoothness of the characteristic functions: the characteristic
function may not be smooth for fat tailed probability density
functions having diverging moments.  Dissipation in the Navier-Stokes
equations, of course, makes kinetic energy decay and adds an element
of exponential damping to the Fourier transformed equation.  We can
use this damping to control and bound the growth of suitable norms.
For convenience we chose to study the $L^1$ norms involving the
Fourier transform of $u$.  The bounds then allow us to obtain
$L^\infty$ estimates for $u$ and its derivatives.  Hence, we actually
work in the classical H\"older space rather than in Sobolev space.

We first briefly sketch the outline of the proof.  We take the Fourier
transform of the Navier-Stokes equation and then bound the rate of
growth of the Fourier transform $\tilde u(k,.)$ of $u(x,.)$, where $k$
is the Fourier conjugate of $x$.  Using this inequality, we then bound
the ``moments'' of the absolute value of $\tilde u$.  These
``moments'' are defined in a manner analogously to statistical moments
of probability density functions.  Conservation of energy plays an
explicit and important role in the argument, because it leads to
estimates of lower order derivatives of $u$ in terms of higher order
derivatives, which leads to the final result.

One aspect of the problem which becomes more clear in this argument
concerns the crucial role played by scale invariance.  We will show in
the concluding remarks that the critical value $\alpha_{\rm \mbox{\tiny
  \sc l}} (n)=(2+n)/4$ for the marginal case corresponds to the scenario
where the dissipation effects and the nonlinear effects (or nonlocal
effects in the Fourier domain) have the same scaling.  The number of
dimensions $n$ enters the picture in the step involving conservation
of energy.  This longer proof provides additional and complementary
insight into how and why conservation of energy has different effects
for different $n$. (In fact, both the divergence-free condition of
incompressibility and conservation of energy become less relevant as
$n$ tends to infinity.)

In what follows $u_i$ denotes the Cartesian component $i=1,2,\ldots n$
of $u$ and $|k|$ is the length of the $n$-dimensional vector $k$.

\begin{proposition}\label{prop-main}
The
Fourier transform $\tilde u_i(k,.)$ of $u_i(x,.)$ 
  satisfies
\bea
 \deriv{t} \big\| {|k|^m \tilde{u}_i }\big\|_{L^1}    &\leq&    
 \sum_j^n  
     \sum_\ell^m  {m \choose \ell}
\big\|{
|k|^\ell \tilde{u}_j} \big\|_{L^1}  \big\|{|k|^{m-\ell+1} \tilde{u}_i} \big\|_{L^1}    \nonumber \\ & &  -  \big\| {|k|^{{2\alpha}+m},{\tilde{u}_i}} \big\|_{L^1} + C_{i,m}(t) \;\;\label{moments} .
  \eea 
for any integer $m\geq 0$, 
where $C_{i,m}$ depend only on pressure $p$.
\end{proposition}

\begin{proof}[Proof of Proposition \ref{prop-main}]

Fourier transforming Equation  (\ref{eq-ns}) to eliminate the space derivatives, we obtain
\bea  \pderiv{t} \tilde{u}_i({
k},t) + \sum_j^d \tilde{u}_j({ k},t ) * \big({\mathtt i}k_j \tilde{u}_i({
k},t)\big)  = - |k| ^{2\alpha}
\tilde{u}_i({ k},t)
 - {\mathtt i}k_i \tilde p({k},t) 
  \label{eq-fourier} ~
  \eea
with the {nonlinear term} becoming a nonlocal convolution
term in Fourier space.  {Here } \mbox{${ \mathtt
    i}=\sqrt{-1}$}, whereas $i$ denotes a component index.  We now
note that for any complex function $g(t)$, 
if \mbox{${d}g/{d}t = A(t) - B(t)
  g(t)$} 
and $B\geq0$ then \mbox{$ \frac{d}{dt}|g| \leq |A| -B |g |$.}
Identifying $g$ with $\tilde u_i$
and $B$ with $|k|^{2\alpha}$, we get from Equation  (\ref{eq-fourier}),
\bea ~~~~ ~~ \pderiv{t} |\tilde{u}_i({k},t)| \leq \bigg|\sum_j
\tilde{u}_j({ k},t ) * \big({\mathtt i}k_j \tilde{u}_i({ k},t)\big) \bigg|
 - |k|^{2\alpha} |\tilde{u}_i({k},t)| 
+ | k_i \tilde p({k},t)|
\eea
Multiplying by $|k|^m$ ($m=0,1,2,\ldots$) and integrating out $k$ over
the entire Fourier domain ${\Bbb R}^n$ , we get,
\bea 
\deriv{t} 
\int_{{\Bbb R}^n}
{ dk}~ |k|^m |\tilde{u}_i({k},t)|
&\leq& \int_{{\Bbb R}^n}   { dk} ~ |k|^{m} \sum^n_j
\bigg|\tilde{u}_j({k},t )*\big(k_j \tilde{u}_i({k},t)\big)
\bigg| \nonumber \\ &~& \quad - \int_ {{\Bbb R}^n} {dk}~ |k|^{{2\alpha}+m} |\tilde{u}_i({k},t)| \nonumber \\ &~&
\quad + \int_{{\Bbb R}^n}  {dk} ~ |k|^m \big(
|k_i \tilde p({k},t)| \big) \label{moment-long} \;\; .
\eea
Hence,%
\bea
\deriv{t} \big\|{|k|^m \tilde{u}_i} \big\|_{L^1}     &\leq&    
\sum_j  
\int {dk}'
\big\|{|{k}+{k'}|^{m} {\tilde{u}_j}} \big\|_{L^1}   
 ~~\big |k_j' \tilde u_i({k'},t) \big |  \nonumber \\
& & 
 -  
\big\| {|k|^{{2\alpha}+m}{\tilde{v}_i}} \big\|_{L^1}   +C_{i,m}(t)~  \nonumber \\
  &\leq&
\sum_j  
  \int { dk}'
     \sum_\ell^m  {m  \choose  \ell} \bigg[ \nonumber \\ & & 
\big\|{|k|^\ell \tilde{v}_j} \big\|_{L^1}  ~~
 |k'|^{m-\ell}
\big | k_j'  \tilde{u}_i(  {k'},t)    \big | \bigg]   
\nonumber \\
& & -  
\big\| {|k|^{{2\alpha}+m}{\tilde{v}_i}} \big\|_{L^1}   + C_{i,m}(t)~   \nonumber
\eea
where $
 C_{i,m}(t) = \big\| {|k|^{m+1} \tilde p} \big\|_{L^1}  
$.
The final integration leads to 
\bea
 \deriv{t} \big\| {|k|^m \tilde{u}_i }\big\|_{L^1}    &\leq&    
 \sum_j^n  
     \sum_\ell^m  {m \choose \ell}
\big\|{
|k|^\ell \tilde{u}_j} \big\|_{L^1}  \big\|{|k|^{m-\ell+1} \tilde{u}_i} \big\|_{L^1}    \nonumber \\ & &  -  \big\| {|k|^{{2\alpha}+m},{\tilde{u}_i}} \big\|_{L^1} + C_{i,m}(t) \;\;\label{moments-qq} .
  \eea 

\end{proof}

The pressure
satisfies an elliptic Poisson equation, which we can invert to obtain,
\bel{eq-p} 
p= -\Delta^{-1} {\rm Tr}  ~ (\nabla u)^2 \;\; .
\ee
It corresponds to the purely non-divergence-free component of the
nonlinearity.  Henceforth, we ignore $C_{i,m}$ since pressure plays no
relevant role, as discussed earlier in the context of Leray
projections.  

Proposition \ref{prop-main} allows us to control the growth of certain
norms. To prove the existence of smooth solutions, we need the term
with ${2\alpha}$ in the exponent to dominate over the sums of product
term (i.e., the inertial or nonlinear term). If this condition is
satisfied, then the norms remain bounded.  We next develop a few
useful lemmas, which we prove for completeness.

\begin{lemma}\label{lemma-1}
 Let $\tilde f$ denote the 
Fourier transform of a function $f: {\bf R^n \rightarrow R^n} $.
If $f\in C^m$, then 
\mbox{$ \bigg\| ~|k |^\ell \tilde  f \bigg\|_{L^1}  \lesssim  \bigg\| ~|k|^m \tilde f~ \bigg\|_{L^1} ^{\frac{\ell+n/2}{m+n/2}}$} for $\ell\leq m$.
\;\; .
\end{lemma}

\begin{proof}

Scale transformations by a factor $\lambda$ show that the inner product
\bea	
\exval{f,f} &=& \exval{\tilde f(k), \tilde f(k)}\\
&=& \exval{\lambda^{-n/2}\tilde f(k/\lambda), \lambda^{-n/2}\tilde f(k/\lambda)}
\;\; ,
\eea
so that $\tilde f(k)$ and $\tilde f_\lambda(k) \equiv  \lambda ^{-n/2} \tilde f(k/\lambda)$
have the same inner product in $L^2=H^0$. Now 
\bea
\bigg\| ~ |k|^\ell \tilde f_\lambda \bigg\|_{L^1} &=& \lambda^{\ell+n/2} \bigg\| ~ |k|^\ell \tilde f \bigg\|_{L^1} \\
\bigg\| ~ |k|^m \tilde f_\lambda \bigg\|_{L^1} &=& \lambda^{m+n/2} \bigg\| ~ |k|^m \tilde f \bigg\|_{L^1}
\eea
So for sufficiently large $\lambda$ we get 
$  \bigg\| ~|k|^\ell \tilde f_\lambda ~ \bigg\|_{L^1} <  \bigg\| ~|k|^m \tilde f_\lambda ~ \bigg\|_{L^1}$.
Moreover 
\be
\frac
{  \bigg\| ~|k|^\ell \tilde f ~
 \bigg\|_{L^1} ^{1/(\ell+n/2)}}
{  \bigg\| ~|k|^m \tilde f ~
 \bigg\|_{L^1} ^{1/(m+n/2)}}
=
\frac
{  \bigg\| ~|k|^\ell \tilde f_\lambda ~
 \bigg\|_{L^1} ^{1/(\ell+n/2)}}
{  \bigg\| ~|k|^m \tilde f_\lambda ~
 \bigg\|_{L^1} ^{1/(m+n/2)}}
\ee
 from 
which it follows that 
$    \big\| ~|k|^\ell \tilde f~ \big\|_{L^1}  < C  \bigg\| ~|k|^m \tilde f~ \bigg\|_{L^1} ^{\frac{\ell+n/2}{m+n/2}}$
for some finite $C$, proving the lemma.

\end{proof}

\begin{lemma}   
Let $\beta=(\beta_1,\beta_2,\ldots, \beta_n) $ be a multi-index and
let $|\beta|$ denote the sum of the components, according to
convention.  If $\bigg\| ~|k|^{|\beta|} \tilde f(k)
\bigg\|_{L^1}<+\infty$ then $f \in C^{|\beta|}$.

\label{lemma-2}
\end{lemma}

\begin{proof}

\bea 
(2\pi)^{n/2}
\big|  
 f^{(\beta)} (x)~\big| \nonumber   &\leq&
\int _{\infty}^\infty    
|{\mathtt i} k|^{|\beta|} 
~ |\tilde f(k)| ~ | e^{{\mathtt i} k \cdot x} | dk 
\;  , \\
(2\pi)^{n/2} \| f^{(\beta)}\|_ {L ^ \infty} &\leq& \bigg\| ~|k|^{|\beta|} \tilde f(k) \bigg\|_ {L ^ 1}
\eea
Hence, $f$ belongs to the Sobolev space $W^{|\beta|,\infty}$,
consequently to the classical 
H\"older space $C^{|\beta|}$.

\end{proof}

\begin{proof}[Proof of Theorem \ref{tm-1}]

We first use Lemma \ref{lemma-1} to bound the
norms in the inertial terms:
\bea
\big\|{
|k|^\ell \tilde{u}_j} \big\|_{L^1}  & \lesssim&  
\bigg\|{
|k|^{m+{2\alpha}} \tilde{u}_j} \bigg\|^{\frac{\ell+n/2}{{2\alpha}+m+n/2}}_{L^1}  
\label{eq-nl1}
\\
\big\|{
|k|^{m-\ell+1} \tilde{u}_i} \big\|_{L^1}  & \lesssim&  
\bigg\|{
|k|^{m+{2\alpha}} \tilde{u}_i} \bigg\|^{\frac{m-\ell+1+n/2}{{2\alpha}+m+n/2}}_{L^1}  
\label{eq-nl2}
\eea

From Proposition \ref{prop-main}, we retain control so long as the sum
of the exponents on the right hand sides of (\ref{eq-nl1}) and
(\ref{eq-nl2}) is less than unity.  The condition to bound the rate of
growth of the relevant $L^1$ norms then becomes
\be
\frac
{m+1+n}
{2\alpha+m+n/2} <  1
\ee
from which we recover the condition (\ref{eq-alpha-condition}). 
Lemma \ref{lemma-2} 
guarantees global solvability,
completing  the proof.

\end{proof}

\section{Concluding remarks}

Attention is traditionally given in the case $n=3$ to vorticity.
In
our proofs the key quantity of interest in the Navier-Stokes system
clearly is the energy.  As noted by Tao~\cite{tao1,tao2}, the energy
is a significant globally controlled coercive quantity which for
$\alpha=1$ is critical for $n=2$ but supercritical for $n>2$.
However, we note that for $\alpha\geq \alpha_{\rm \mbox{\tiny \sc l}} (n)$, the
energy remains both coercive and either critical ($\alpha=\alpha_{\rm
  \mbox{\tiny \sc l}} (n)$) or subcritical ($\alpha>\alpha_{\rm \mbox{\tiny \sc l}} (n)$).  We
can state these ideas formally:

\begin{theorem}
The kinetic energy $ {1\over 2 }\|u(.,T) \|^2_{L^2}$ is critical under
scale transformations only for $\alpha=\alpha_{\rm \mbox{\tiny \sc l}} (n)$.
\end{theorem}

\begin{proof}
If $u$ is a solution, then
\be
u_\lambda(x,t) = {1\over\lambda}
u\left( \frac{x}{\lambda^{\frac{1}{2\alpha-1}}}   ,\frac{t}
 {\lambda^{\frac{2\alpha}{2\alpha-1}}}
\right)
\ee
is also a solution~\cite{epl}.
The energy $E_\lambda$ of the new solution 
scales to 
\bel{eq-es}
E_\lambda= \lambda^{2-\frac n{(2\alpha-1)}}E
.
\ee
We thus see that the value $\alpha_{\rm \mbox{\tiny \sc l}} (n)=(2+n)/4$
corresponds to a critical value, for which the energy remains
invariant under scale transformations, i.e. $E_\lambda=E$.
\end{proof}

\begin{remark}
We briefly comment on the marginal case $\alpha=\alpha_{\rm \mbox{\tiny \sc l}} $.
We know from Proposition \ref{prop-main} and Equations (\ref{eq-nl1})
and (\ref{eq-nl2}) that the problem for the marginal case arises
because the nonlinear and dissipation terms scale identically and thus
balance each other, so that it is not immediately clear which term
dominates.  Any small additional consideration can tip the balance.
To address this issue, we recall that
\be
\|u(.,T_2)\|^2_{L^2}  < \|u(.,T_1)\|^2_{L^2}
\ee
for $T_2>T_1$, i.e. the kinetic energy is a strictly decreasing
function of time. This reduction in kinetic energy, insignificant for
the non-marginal cases $\alpha<\alpha_{\rm \mbox{\tiny \sc l}} (n)$, can by itself
rule out any blow-up for the marginal case, as suggested by the
criticality condition in Equation (\ref{eq-es}).

We could, alternatively, describe the situation from a different
perspective, viz., blow-ups are energetically forbidden for $\alpha>
\alpha_{\rm \mbox{\tiny \sc l}} (n)$, but not for $\alpha<\alpha_{\rm \mbox{\tiny
  \sc l}} (n)$.  The question is what happens exactly at the critical point.
In principle, if the kinetic energy were a constant of the motion,
then we would not be able to rule out the possibility that almost all
the energy might become concentrated into ever smaller regions of
space and at an increasingly faster rate.  In this scenario, we would
not be able to rule out finite time singularities.  However, kinetic
energy is a strictly decreasing function of time, hence in the time
that it takes to concentrate the energy into a smaller region, enough
energy will have dissipated to prevent the formation of singularities
in finite time, because not only do dissipation and nonlinearity have
the same scaling for $\alpha=\alpha_{\rm \mbox{\tiny \sc l}} (n)$, but in addition
the kinetic energy decays asymptotically to zero at large times.

\end{remark}

Finally, we comment on the hypothetical worst-case scenario where the
corresponding Euler equations for inviscid flows lead to blow-ups.  No
globally controlled coercive critical or subcritical quantity is known
to exist for $\alpha<\alpha_{\rm \mbox{\tiny \sc l}} $, therefore we
are forced to consider the possibility that some initial conditions
may very well lead to blow-ups in the Navier-Stokes case for
$\alpha<\alpha_{\rm \mbox{\tiny \sc l}} $ in this hypothetical
scenario.

For $n=3$, the Euler equations may allow for the uncontrolled growth
of the vorticity $\omega=\nabla\times u$ due to the vorticity
stretching mechanism~\cite{gibbon,eul1,constantin}.  The
enstrophy \be {1\over 2} \| \omega\|^2_{L^2} = {1\over 2}
\exval{\nabla \times u,\omega} ={1\over 2} \exval{u,\nabla\times
  \omega}=-{1\over 2} \exval{u,\Delta u}= {1\over 2} \|\nabla u \|_{L^2} \ee
can thus conceivably blow up because it also is not conserved:
\be {1\over 2} \deriv{t} \| \omega\|^2_{L^2} =
\exval{\omega\cdot \nabla u,\omega} \;\; .  \ee 
If it could be shown that a blow-up is possible for the Euler
equation, then there may indeed be no way to regularize this
worst-case scenario with a dissipation having $\alpha<\alpha_{\rm
  \mbox{\tiny \sc l}} (3)$.
In this context, we recall the statement of
Peter Constantin~\cite{constantin}: ``It is no exaggeration to say
that the Euler equations are the very core of fluid dynamics.''

The same argument can be generalized to $n>3$: if $\| \nabla
u\|^2_{L^2} $ blows up in the Euler case, no way is (currently) known
to regularize solutions via dissipation for $\alpha<\alpha_{\rm
  \mbox{\tiny \sc l}} (n)$.  
Although the question remains open, we speculate that the $\| \nabla
u\|^2_{L^2} $ can blow up in solutions of the Euler equations in three and higher dimensions.
This intuition leads to our final statement:

\begin{conjecture} 
The Ladyzhenskaya-Lions exponent $\alpha_{\rm \mbox{\tiny \sc l}} (n)$
is critical in the sense that it separates two regions.  Solutions of
the hyper-dissipative Navier-Stokes equations in dimensions $n\geq 3$
remain smooth for $\alpha \geq \alpha_{\rm \mbox{\tiny \sc l}} (n)$,
whereas for any \mbox{$\alpha< \alpha_{\rm \mbox{\tiny \sc l}} (n)$},
finite time singularities of $\| \nabla u\|^2_{L^2} $ are possible.
\end{conjecture}

%

%



\begin{thebibliography}{99}
 
\bibitem{book-funcanal}
%
Y. M. Berezansky, Zinovij G. Sheftel
and Georgij F. Us,
{ Functional Analysis, Vol. 2} (Birkhäuser, Berlin, 1996).



\bibitem{eul1}
P. Constantin, Ch. Fefferman and  A. Majda. Geometric constraints on
     potentially singular solutions for the 3D Euler equation, Comm. Partial
     Differential Equations {\bf 21} (1996) 559­-571.


\bibitem{constantin}
%
%
%
%
%
P. Constantin,
On the Euler equations of incompressible fluids,
Bulletin of the American Mathematical Society {\bf 44} (2007)
603--621.



\bibitem{gibbon}
J. D. Gibbon, The three-dimensional Euler equations: Where do we stand?
Physica D {\bf 237} (2008) 1894--1904. 

\bibitem{new} 
Jean-Luc Guermond  and Serge Prudhomme,
%
%
{  Mathematical aanalysis of a spectral hyperviscosity
LES model for the simulation of turbulent flows}.
ESAIM: M2AN 
%
{\bf 37} (2003)  893--908.
%
\\
DOI: 10.1051/m2an:2003060


\bibitem{katz}
{ N. H. Katz and  N. Pavlovi\'c},
{ A Cheap Caffarelli-Kohn-Nirenberg inequality
for the Navier-Stokes equation with hyper-dissipation}.
{Geom.  Funct.  Anal.}
  {\bf 12} (2002)  { 355--379}.
%
  %
%


\bibitem{ms} J. Mattingly and Y. Sinai, { An elementary proof of the
  existence and uniqueness theorem for the Navier-Stokes equation}.
%
%
  {Commun. Contemp. Math.}
   {\bf 1} (1999)  {497--516}.
%
%
%


 
 \bibitem{sk1}
O. Ladyzhenskaya, { The Mathematical Theory of Viscous Incompressible Flows} 
(Gordon and Breach, New York,
1969).
 
 \bibitem{sk2}
O. Ladyzhenskaya, 
{  Mathematical analysis of
Navier-Stokes equations for incomressible
liquids.  }{ 
Annu. Rev. Fluid Mech.} 
{\bf 7}  (1975) 
249--272.



\bibitem{lions1} 

J.-L. Lions,  { 
Quelques r\'esultats d'existence dans des \'equations
aux d\'eriv\'ees partielles non lin\'eaires.}
{Bull. Soc. Math. France}
{\bf 87} (1959)  245--273.
%
%

\bibitem{lions2} J.-L. Lions, { Quelques m\'ethodes de r\'esolution des
  probl\`emes aux limites non lin\'eaires} 
%
(Dunod, Paris, 1969).


\bibitem{sokolov}
%
%
%
%
%
%
G. Radons, R. Klages and  I. M. Sokolov, Anomalous Transport
(Wiley-VCH, Berlin, 2008).




\bibitem{tao1} T. Tao, { A quantitative formulation of the global
  regularity problem for the periodic Navier-Stokes equation}.
%
Dyn. Partial Diff. Equat.  {\bf 4}, 293--302 (2007).



\bibitem{tao2} T. Tao, 
{ Why global regularity for Navier-Stokes is hard.}  \\
URL: 
http://terrytao.wordpress.com/2007/03/18 

\bibitem{epl} G. M. Viswanathan and T. M. Viswanathan, 
{ Spontaneous symmetry breaking and finite-time singularities in
d-dimensional incompressible flows with fractional dissipation}.
EPL {\bf 84} (2008)
50006.
%



\end{thebibliography}
\end{document}